\documentclass[10pt,a4paper,reqno]{amsart}
\usepackage[utf8]{inputenc}

\usepackage[makeroom]{cancel}
\usepackage{float}
\usepackage{mathabx}
\usepackage{amsmath}
\usepackage{amsthm}
\usepackage{amssymb}
\usepackage{amscd}
\usepackage{amsfonts}
\usepackage{amsbsy}
\usepackage{mathrsfs}
\usepackage{graphicx}
\usepackage[dvips]{psfrag}
\usepackage{array}
\usepackage{xcolor}
\usepackage{epsfig}
\usepackage{url}
\usepackage{hyperref}
\usepackage{enumitem}
\usepackage{overpic}
\usepackage{epstopdf}
\usepackage{apptools}
\usepackage{amsmath}
\usepackage{diagbox}

\newcommand{\x}{\mathbf{x}}

\def\e{\varepsilon}

\newtheorem {teo} {Theorem}

\newtheorem {propo} {Proposition}

\newtheorem {obs} {Remark}

\newtheorem{main}{Theorem}

\title[Normally hyperbolic limit tori near monodromic singularities in 3D]{Normally hyperbolic limit tori near monodromic singularities in 3D polynomial vector fields}

\author[ L. Q. Arakaki]{Lucas Q. Arakaki$^*$}
\address{Departamento de Matem\'{a}tica, Instituto de Matem\'{a}tica, Estatística e Computa\c{c}\~{a}o Cient\'{i}fica (IMECC), Universidade
Estadual de Campinas (UNICAMP), Rua S\'{e}rgio Buarque de Holanda, 651, Cidade Universit\'{a}ria Zeferino Vaz, 13083--859, Campinas, SP, Brazil.} 
\email{larakaki@unicamp.br}

\author[ L. F. S. Gouveia]{Luiz F. S. Gouveia}
\address{Departamento de Matem\'{a}tica, Instituto de Matem\'{a}tica, Estatística e Computa\c{c}\~{a}o Cient\'{i}fica (IMECC), Universidade
	Estadual de Campinas (UNICAMP), Rua S\'{e}rgio Buarque de Holanda, 651, Cidade Universit\'{a}ria Zeferino Vaz, 13083--859, Campinas, SP, Brazil.} 
\email{lgouveia@unicamp.br}

\author[D. D. Novaes]{Douglas D. Novaes}
\address{Departamento de Matem\'{a}tica, Instituto de Matem\'{a}tica, Estatística e Computa\c{c}\~{a}o Cient\'{i}fica (IMECC), Universidade
Estadual de Campinas (UNICAMP), Rua S\'{e}rgio Buarque de Holanda, 651, Cidade Universit\'{a}ria Zeferino Vaz, 13083--859, Campinas, SP, Brazil.} \email{ddnovaes@unicamp.br}

\makeatletter 
\@namedef{subjclassname@2020}{%
	\textup{2020} Mathematics Subject Classification}
\makeatother
\thanks{$^*$ Corresponding author.}



\allowdisplaybreaks

\begin{document}

\begin{abstract}
We investigate the maximal number $N_h(m)$ of normally hyperbolic limit tori in three-dimensional polynomial vector fields of degree $m$, which extends the classical notion of Hilbert numbers to higher dimensions. Using recent developments in averaging theory, we show the existence of families of vector fields near monodromic singularities, including both Hopf-zero and nilpotent-zero cases, that exhibit multiple nested normally hyperbolic limit tori. This approach allows us to establish improved lower bounds: $N_h(2) \geq 3$, $N_h(3) \geq 5$, $N_h(4) \geq 7$, and $N_h(5) \geq 13$, which are currently the best available in the literature. 
Furthermore, these bounds are extended using the strict monotonicity of the function $N_h(m)$ and a recursive construction inspired by the Christopher–Lloyd method, leading to new estimates for higher degrees which improves all the previously known results.
\end{abstract}

\maketitle


\section{Introduction}

Motivated by Hilbert's 16th Problem, the quest to determine the finiteness of the so-called Hilbert number $H(m)$, $m\geq 2$---the maximum number of limit cycles that planar polynomial vector fields of degree $m$ can exhibit---has driven research for over a century. This problem has been one of the main catalysts for developments in the qualitative theory of differential equations. 

Most advances concerning $H(m)$ have centered on establishing lower bounds. Early contributions by Otrokov~\cite{O54}, Il’yashenko~\cite{I91}, and Basarab-Horwath \& Lloyd~\cite{BL88} demonstrated that $H(m)$ grows at least quadratically with $m$. This was later refined by Christopher \& Lloyd~\cite{ChristopherLloyd}, who showed that $H(m)$ grows at least as fast as $m^2 \log m$ (see also~\cite{lvarez2020,HL12,IS00}). In recent decades, substantial progress—particularly in the study of the cyclicity problem, both local and global—has led to further improvements in the known lower bounds for $H(m)$ at low degrees. The best-known local lower bounds have been obtained through a weaker version of Hilbert’s 16th Problem, proposed by Arnol’d~\cite{Arnold1977}, which focuses on the number $M(m)$ of limit cycles that bifurcate from the period annulus of a Hamiltonian system under polynomial perturbations of degree $m$. A classical result in this direction is due to Bautin~\cite{Bautin54}, who proved that $M(2) = 3$. More recently, improvements have been obtained for $m > 2$, with current best-known results including $M(3) \geq 12$ and $M(4) \geq 21$\cite{GouveiaGine}, $M(5) \geq 33$, $M(7) \geq 76$, $M(8) \geq 76$, $M(9) \geq 88$\cite{GouveiaTorre21}, and $M(6) \geq 48$~\cite{BastosBuzzi_cotas}. As for global estimates, the best known lower bounds for the Hilbert number include $H(2) \geq 4$\cite{ChenWang_cota}, $H(3) \geq 13$\cite{LiuLi_cota}, and, for higher degrees, $H(4) \geq 28$, $H(5) \geq 37$, $H(6) \geq 53$, $H(7) \geq 74$, $H(8) \geq 96$, $H(9) \geq 120$, and $H(10) \geq 142$~\cite{ProhensTorre_cotas}.

It is worth noting that the limit cycles associated with the aforementioned local lower bounds arise from perturbations of a \emph{Hopf equilibrium}. Another important type of singularity where the cyclicity problem can also be studied is the \emph{nilpotent equilibrium}. In this direction, the works \cite{Alvarez1,Alvarez2,GarciaNil} provide a systematic investigation into the cyclicity of nilpotent singular points. However, the number of limit cycles obtained in these studies is significantly smaller than those found in the case of Hopf singularities.


In higher dimensions, isolated invariant tori play a role analogous to that of limit cycles in the plane, providing a natural framework for extending classical questions related to existence, stability, number, and distribution. From this perspective, the authors of~\cite{NP25} proposed an analogue of Hilbert's 16th Problem in three-dimensional space: to investigate the maximal number $ N(m) $ of isolated invariant tori in polynomial vector fields of degree $ m $. More precisely, for a polynomial vector field 
\[
X = P(x,y,z)\,\partial_x + Q(x,y,z)\,\partial_y + R(x,y,z)\,\partial_z,
\]
let $ \tau(P, Q, R) $ denote the number of isolated invariant tori, hereafter referred to as \emph{limit tori}, in its phase space. The corresponding maximal number is then defined as
\[
N(m): = \sup \left\{ \tau(P, Q, R) : \deg(P), \deg(Q), \deg(R) \leqslant m \right\}.
\]

Among the compact invariant manifolds of vector fields, \emph{normally hyperbolic} ones are of particular interest due to their robust dynamical properties, such as persistence under small perturbations and well-defined stable and unstable directions (see, for instance, \cite{Fenichel1971,hirschpughshub,wigginsnormally}). In this context, a normally hyperbolic limit torus serves as a natural higher-dimensional analogue of a hyperbolic limit cycle. This analogy motivates the following definition:
\[
N_h(m): = \sup\left\{ \tau_h(P, Q, R) : \deg(P), \deg(Q), \deg(R) \leqslant m \right\},
\]
where $ \tau_h(P, Q, R) $ denotes the number of normally hyperbolic limit tori of the vector field $ X = P \partial_x + Q \partial_y + R \partial_z $. It is evident that $ N_h(m) \leqslant N(m) $.

Analogously to the Hilbert number $H(m)$, it is an open question whether $N(m)$ or $N_h(m)$ are finite.

Unlike limit cycles, for which numerous analytical tools exist, the identification of limit tori in differential systems remains a challenging task due to the scarcity of available analytical methods. Among the few effective approaches, averaging theory stands out as a particularly powerful technique. Initially developed by Krylov, Bogoliubov, and Mitropolski~\cite{BM,BK}, and later extended by Hale~\cite{haleinvariant,haleoscillations}, this method has been further generalized in recent works~\cite{Novaes2023MathAnn,Pereira23Mechanism} to encompass broader classes of differential equations of the form $\dot{\x}=\e F(t,\x,\e)$ . These advancements include important results concerning the existence of invariant tori, their stability properties, and the qualitative nature of the dynamics they support. A key feature of this approach is the reduction of the non-autonomous system $\dot{\x}=\e F(t,\x,\e)$ to a simpler autonomous guiding system:
\begin{equation}\label{introgs}
\dot{\x} = \dfrac{1}{T} \int_0^T F_1(t, \x) \, dt,
\end{equation}
in which the existence of hyperbolic limit cycles implies the existence of normally hyperbolic invariant tori in the original system. This connection offers a concrete and computationally feasible criterion for detecting such tori, making the method highly valuable in both theoretical and applied contexts. Indeed, it has proven effective in analyzing the emergence of invariant tori across a variety of applied models (see, for example,~\cite{Candido_2020Rossler,CandidoVallsVDPD,Messias2024}). It is important to note that the differential equation~\eqref{introgs} corresponds to the guiding system when the average on its right-hand side is not identically zero. In cases where this average vanishes identically, a higher-order analysis can be employed to define the guiding system, and the same result remains valid. These concepts will be discussed in detail in Section~\ref{Sec:Prelim}.

This method was employed in~\cite{NP25} to establish the first general lower bounds for $ N_h(m) $ for all $ m \geqslant 6 $, in particular, it was shown that $ N_h(6) \geqslant 4 $. For lower degrees, secondary Hopf bifurcation \cite{Candido2020JDE} has been used to detect invariant tori in specific systems. In~\cite{Candido_2020Rossler}, one such torus was found in a particular quadratic R\"ossler system, while in~\cite{CandidoVallsVDPD}, two tori were identified in a cubic system. Since the invariant tori arising from secondary Hopf bifurcations are normally hyperbolic (see, for instance,~\cite{PN25b}), these results imply that $ N_h(2) \geqslant 1 $ and $ N_h(3) \geqslant 2 $.
 
It is worth noting that if there exists a three-dimensional vector field of degree $ m_0 $ with $ \tau_0 $ limit tori, then, as shown in \cite{NP25} via a Christopher--Lloyd-type transformation (see \cite{ChristopherLloyd}), one obtains  
\begin{equation}\label{CLmeth}
N\big(2^k(m_0+2)-2\big)\geq 8^k \tau_0.
\end{equation}  
This inequality also holds for $N_h$, provided that the $\tau_0$ limit tori are normally hyperbolic. Consequently, lower bounds established for low degrees can be systematically extended to higher degrees. We anticipate that, here, the averaging theory will be employed to establish new lower bounds for $N_h(m)$ with $m \leq 5$, which will then be extended to higher degrees via inequality~\eqref{CLmeth}.

\subsection{Goal and main results} Motivated by developments in the planar case, specifically, the efforts to establish and refine lower bounds for $H(m)$, this work aims to improve the known lower bounds for $N_h(m)$ for all $m \geq 2$. Our approach is based on recent advances in the averaging method, which we apply to detect limit tori in three-dimensional vector fields that are close to ones exhibiting monodromic behavior around the z-axis. Specifically, we focus on vector fields of the form
\begin{equation}\label{eq:01}
X = X_0 + \varepsilon X_1(\mu) + \varepsilon^2 X_2(\mu),
\end{equation}
where $ \mu \in \Lambda \subset \mathbb{R}^M $, $ \varepsilon > 0 $ is a small parameter, and $ X_1 $, $ X_2 $ are polynomial vector fields of degree $ m $. The unperturbed part is given by 
\[
 X_0 = -y\,\partial_x + x^{2n-1}\,\partial_y,
 \] with $ n \geqslant 1,$ and has the $ z $-axis as an invariant line of singularities, around which nearby orbits exhibit rotational (monodromic) behavior. This follows from the fact that each plane $ z = z_0 $ is invariant under the flow of $ X_0 $, and on each such plane, the origin is a monodromic singularity. When $ n = 1 $, the origin of $ X_0 $ is a \emph{Hopf-zero} singularity; for $ n > 1 $, we refer to it as a \emph{nilpotent-zero} singularity. In this latter case, $ n $ is called by \emph{Andreev number} associated with the singularity. 
 
With this approach we are able to, firstly, produce several examples of $3$D vector fields with degree  $ m \in \{2,3,4,5\} $  for which multiple nested limit tori unfold from the monodromic singularity at the origin. As a consequence, we obtain our main result:
\begin{main}\label{thm:main}
$N_h(2)\geqslant 3$, $N_h(3)\geqslant 5$, $N_h(4)\geqslant 7$, and $N_h(5)\geqslant 13$.
\end{main}
The above lower bounds for $N_h$ are the best known in the literature. 

We summarize the results of our investigation in Table~\ref{tab:toriinvestigation}, which presents the highest number of limit tori detected for each case. Here, $m$ denotes the degree of the vector field, and $n$ refers to the Andreev number associated with the origin (with $n=1$ corresponding to the Hopf-zero case). All entries in Table~\ref{tab:toriinvestigation} from the second row onward (i.e., for $m \geq 3$) were obtained through a first-order analysis using the averaging method. In contrast, the improvement for $m = 2$ was achieved via second-order averaging, yielding the new lower bound for quadratic systems $N_h(2) \geqslant 3$.

\begin{table}[ht]
\begin{center}
\begin{tabular}{|c|c|c|c|}
\hline
\diagbox{$m$}{$n$} & 1 & 2 & 3\\
\hline
    2  & 3 & - & -\\
\hline
    3  & 3 & 5 & -\\
\hline
    4  & 6 & 7 & -\\
\hline
    5  & 9 & 11 & 13\\
\hline
    \end{tabular}
    \smallskip
    \caption{Number of limit tori obtained for vector fields \eqref{eq:01} with degree $m$ and Andreev number $n$. }
    \label{tab:toriinvestigation}
\end{center}
\end{table}

Notably, and in contrast to the planar case, the best lower bounds for $N_h(m)$ were obtained from nilpotent-zero singularities. This is a consequence of the averaging method, which, broadly speaking, reduces the problem of identifying limit tori in the original system to that of finding limit cycles in the averaged system. Importantly, even when the original system has a nilpotent-zero singularity, the corresponding averaged system still features Hopf-type equilibria, enabling the application of bifurcation techniques to detect limit cycles, and thus limit tori, in this setting.

Moreover, it was shown in \cite{QDStrictIncrease} that the function $N_h(m)$ is strictly increasing. In view of this result, and taking into account the relationship given in \eqref{CLmeth}, the lower bounds for $N(m)$ established in Theorem~\ref{thm:main} can be extended to higher degrees, as presented in Table~\ref{higherdegrees}. These improved lower bounds surpass those previously obtained in the literature \cite{Candido_2020Rossler,CandidoVallsVDPD,NP25}. 

\begin{table}[ht]\label{higherdegrees}
\begin{center}
\begin{tabular}{ |c|c|c|c|c|c|c|c|c|c|c|c|c|c|c|c|c| } 
 \hline
 $m$ & 2 & 3 & 4 & 5 & 6 & 7 & 8 & 9 & 10 & 11 & 12 & 13 & 14 & 15 & 16&$\cdots$ \\
 \hline
  old $\tau_m$ & 1 & 2 & 2 & 2 & 8 & 8 & 16 & 16 & 28 & 28 & 37 & 37 & 64 & 64 & 64&$\cdots$\\
 \hline
 new $\tau_m$ & 3 & 5 & 7 & 13 & 24 & 25 & 40 & 41 & 56 & 57 & 104 & 105 & 192 & 193 & 200&$\cdots$\\ 
 \hline
\end{tabular}
\smallskip
\caption{Lower bound $\tau_m$ for the number of normally hyperbolic limit tori in three-dimensional polynomial vector fields of degree $m$. The old lower bounds are taken from \cite{Candido_2020Rossler, CandidoVallsVDPD, NP25}.}
\end{center}
\end{table}

\subsection{Structure of the paper}

This paper is structured as follows:

Section~\ref{Sec:Prelim} reviews key results from bifurcation theory that serve as the foundation for our analysis. In particular, Subsection~\ref{sec:averaging} introduces the fundamental concepts of averaging theory and presents recent results concerning the bifurcation of limit tori. Subsection~\ref{sec:degHopf} focuses on the degenerate Hopf bifurcation and provides useful results related to the computation of Lyapunov quantities.

Section~\ref{Sec:method} develops a general first-order analysis for systems of the form~\eqref{eq:01}, deriving the corresponding guiding system \eqref{introgs} for every degree $m\geq 2$ and Andreev numbers $n \geqslant 1$ satisfying $2n-1\leq m$.

Section~\ref{sec:proof} is dedicated to deriving the results presented in Table~\ref{tab:toriinvestigation} and, consequently, to proving Theorem~\ref{thm:main}. The analysis is divided into several subsections. Subsections~\ref{subsec:m3}, \ref{subsec:m4}, and~\ref{subsec:m5} address the cases $m = 3, 4, 5$, respectively, using the first-order analysis developed in Section~\ref{Sec:method}. Then, Subsection~\ref{Sec:N2>3} employs a second-order analysis to improve the lower bound for the case $m = 2$, which cannot be achieved through first-order analysis alone. The proof of Theorem~\ref{thm:main} follows from Theorems~\ref{N(3)=5}, \ref{N(4)=7}, \ref{N(5)=13}, and~\ref{N(2)=3}.

\section{Preliminary results}\label{Sec:Prelim}

\subsection{Averaging theory}\label{sec:averaging}
In this section, we briefly discuss the averaging method which will be the main tool of the present investigation. For a detailed exposition on this method, we refer the reader to \cite{Pereira23Mechanism}. 

We consider systems of non-autonomous $T$-periodic differential equations, given in the standard form
\begin{equation}\label{eq:Averaging}
\dot{\mathbf{x}}=\sum_{i=1}^{N}\varepsilon^iF_i(t,\mathbf{x};\mu)+\varepsilon^{N+1}\tilde{F}(t,\mathbf{x};\mu,\varepsilon),\quad (t,\mathbf{x};\mu,\varepsilon)\in\mathbb{R}\times D\times\Lambda\times [0,\varepsilon_0],
\end{equation}
where $D$ is a bounded subset of $\mathbb{R}^2$, $\Lambda$ is an open subset of $\mathbb{R}^n$, $\varepsilon_0>0$, $\tilde{F}:\mathbb{R}\times D\times\Lambda\times [0,\varepsilon_0]\to\mathbb{R}^2$ and each $F_i:\mathbb{R}\times D\times\Lambda\to\mathbb{R}^2$, $i\in\{1,\dots,N\}$ is a $C^r$ function, $r>1$ and $T$-periodic in the variable $t$.
The averaging theory provides asymptotic estimates for solutions of differential equations given in the standard form \eqref{eq:Averaging} which are obtained via the \emph{averaged functions} $\mathbf{g}_i:D\times\Lambda\to\mathbb{R}^2$, $i\in\{1,\dots,N\}$. 
These functions have a strong connection to the \emph{Melnikov functions}, which was explored in detail in \cite{DouglasEstrobo}. The Melnikov functions $\mathbf{f}_i$ are given by the coefficients of the power series expansion about $\varepsilon=0$ of the $T$-map of \eqref{eq:Averaging} and can be explicitly computed by the recursive formula
$$\mathbf{f}_i(\mathbf{z};\mu)=\frac{y_i(T,\mathbf{z};\mu)}{i!},$$
where 
\begin{eqnarray}
y_1(t,\mathbf{z};\mu)&=&\int_{0}^{t}F_1(s,\mathbf{z};\mu)ds,\nonumber\\
y_i(t,\mathbf{z};\mu)&=&\int_{0}^{t}\left(i!F_i(s,\mathbf{z};\mu)+\sum_{j=1}^{i-1}\sum_{m=1}^{j}\frac{i!}{j!}\partial_x^{m}F_{i-j}(s,\mathbf{z})B_{j,m}(y_1,\dots,y_{j-m+1})(s,\mathbf{z};\mu)\right)ds,\nonumber
\end{eqnarray}
for $i\in\{2,\dots,N\}$, $\mathbf{z}\in D$ and $B_{j,m}$ being the \emph{partial Bell polynomials}, given by 
$$B_{p,q}(x_1,\dots,x_{p-q+1})=\sum\frac{p!}{b_1!b_2!\cdots b_{p-q+1}!}\prod_{j=1}^{p-q+1}\left(\frac{x_j}{j!}\right)^{b_j},$$
where the sum is taken over all the $(p-q+1)$-tuples of nonnegative integers $(b_1,\dots,b_{p-q+1})$ satisfying $b_1+2b_2+\dots+(p-q+1)b_{p-q+1}=p$ and $b_1+b_2+\dots+b_{p-q+1}=q$.

Elementary formal computations give the following expression for the first two Melnikov functions.
\begin{eqnarray}\label{eq:averageformulas}
\mathbf{f}_1(\mathbf{z};\mu)&=&\int_{0}^{T}F_1(t,\mathbf{z};\mu)dt,\nonumber\\
\mathbf{f}_2(\mathbf{z};\mu)&=&\int_{0}^{T}\left(F_2(t,\mathbf{z};\mu)+\partial_xF_1(t,\mathbf{z};\mu)y_1(t,\mathbf{z};\mu)\right)dt=\nonumber\\
&=&\int_{0}^{T}\left(F_2(t,\mathbf{z};\mu)+\partial_xF_1(t,\mathbf{z};\mu)\int_{0}^{t}F_1(s,\mathbf{z};\mu)ds\right)dt.
\end{eqnarray}

The following result establishes the relationship between the averaged functions and the Melnikov functions.

\begin{propo}[\cite{DouglasEstrobo}, Corollary A]\label{propo:averagedfunctions} Let $\ell\in\{2,\dots,N\}$. If either $\mathbf{f}_1=\cdots=\mathbf{f}_{\ell-1}=0$ or $\mathbf{g}_1=\cdots=\mathbf{g}_{\ell-1}=0$, then $\mathbf{f}_i=T\mathbf{g}_i$ for $i\in\{1,\dots,\ell\}$.
\end{propo}

Now, let $\ell\in\{1,\dots,N\}$ be the index of the first non-vanishing Melnikov function (which by the above result is also the index of the first non-vanishing averaged function) and consider the so-called \emph{guiding system}
\begin{equation}\label{eq:guiding}
\dot{\mathbf{z}}=\mathbf{g}_\ell(\mathbf{z};\mu)=\frac{1}{T}\mathbf{f}_\ell(\mathbf{z};\mu).
\end{equation}

The next two results illustrate the core of the averaging method, which is the fact that the existence of invariant objects of the guiding system \eqref{eq:guiding} imply in the existence of invariant objects in the phase space of system \eqref{eq:Averaging}.

\begin{teo}[{\cite[Theorem A]{LNT2014}}]\label{Teo:classic}
Consider the differential equation \eqref{eq:Averaging}. Suppose that for some $\ell\in\{1,\cdots,N\}$, $\mathbf{f}_1=\cdots=\mathbf{f}_{\ell-1}=0$, $\mathbf{f}_\ell\neq 0$, and that the guiding system \eqref{eq:guiding} has a simple singularity $\mathbf{z}^\ast$. Then, for $\varepsilon>0$ sufficiently small, the differential equation \eqref{eq:Averaging} has an isolated $T$-periodic solution converging to $\mathbf{z}^\ast$ as $\varepsilon\to 0$.
\end{teo}

\begin{teo}[{\cite[Theorem A]{Pereira23Mechanism}}]\label{Teo:InvariantTori}
Consider the differential equation \eqref{eq:Averaging}. Suppose that for some $\ell\in\{1,\dots,N\}$, $\mathbf{f}_1=\cdots=\mathbf{f}_{\ell-1}=0$, $\mathbf{f}_\ell\neq 0$ and that the guiding system \eqref{eq:guiding} has a hyperbolic attracting (resp. repelling) limit cycle $\gamma$. Then, for each $\varepsilon>0$ sufficiently small, the differential equation \eqref{eq:Averaging} has a $T$-periodic solution $\gamma_{int}$ and a normally hyperbolic attracting (resp. repelling) invariant torus in the extended phase space. In addition, the torus surrounds the periodic solution $\gamma_{int}$ and converges to $\gamma\times\mathbb{S}^1$ as $\varepsilon\to 0$.
\end{teo}

Thus, the problem of finding limit tori in the extended phase space of system~\eqref{eq:Averaging} reduces to the search for limit cycles in the phase space of the associated guiding system, which, in our case, will be a planar differential equation. In this context, we recall some fundamental results concerning the unfolding of limit cycles from a Hopf equilibrium in planar systems.

\subsection{Degenearate Hopf Bifurcation}\label{sec:degHopf}

Consider the following family of planar system of differential equations:
\begin{equation}\label{eq:traceHopf}
\begin{array}{l}
\dot{x}=\tau x-\omega y+P(x,y;\mu),\\
\dot{y}=\omega x+\tau y+Q(x,y;\mu),
\end{array}
\end{equation}
where $P,Q$ are polynomials with no constant or linear terms and $\mu\in\Lambda\subset\mathbb{R}^M$, $\tau\in\mathbb{R}$ and $\omega\neq 0$. The origin is a monodromic equilibrium for the above system and for $\tau=0$, we say that it is a \emph{Hopf} equilibrium. This type of equilibrium has been extensively studied in the literature in the context of the bifurcation of limit cycles (see, for instance, \cite{Bautin54,Christopher05Book,TorreParalell,Zoladek95}). 

The main technique consists in the computation of the \emph{Lyapunov coefficients} $v_{2i+1}(\mu)$ which are the odd indexed coefficients of the Poincaré map associated to the origin and verifying some transversality hypothesis that we will explicitly state on the next results. It is important to remark that, since the obtaining of the Lyapunov coefficients are computationally challenging, one can obtain an alternative set of coefficients $L_{i}(\mu)$  which can be computed through an algebraic recursive algorithm \cite{Amarelin,Romanovski}. It is well-known that these two sets of coefficients are equivalent in the context of the cyclicity and center problems. Thus, we shall henceforth refer to the coefficients $L_{i}$ as Lyapunov coefficients.

The algebraic variety defined by $V^c=\{\mu\in\Lambda:L_{i}(\mu)=0,\forall\; i\in\mathbb{N}\}$ is called the \emph{Bautin variety} and is precisely the set of parameters for which the origin of system \eqref{eq:traceHopf}, with $\tau=0$, is a \emph{center}. For every $\mu\in\Lambda\setminus V^c$, the origin of system \eqref{eq:traceHopf}, with $\tau=0$, is a \emph{weak focus}.

The main results on the bifurcation of limit cycles for system \eqref{eq:traceHopf}, whose proofs can be found on \cite{Yu2005} and \cite{GouveiaTorre21}, are the following.

\begin{teo}[Limit cycle bifurcation from a weak focus]\label{Teo:ChristopherFocus}
 Let $L_1(\mu),\dots, L_k(\mu)$ be the first $k$ Lyapunov coefficients associated to the origin of system \eqref{eq:traceHopf} with $\tau=0$. If there exists $\mu^\ast\in\Lambda$ such that $L_1(\mu^\ast)=\dots=L_{k-1}(\mu^\ast)=0$, $L_{k}(\mu^\ast)\neq 0$ and
$${\rm rank}\left[\dfrac{\partial(L_1,\dots,L_{k-1})}{\partial\mu}\right]_{\mu=\mu^\ast}=k-1,$$
then there exist parameter values $(\tau,\mu)$ near $(0,\mu^\ast)$ for which there are $k$ limit cycles in a neighborhood of the origin of system \eqref{eq:traceHopf}.
\end{teo}

\begin{teo}[Limit cycle bifurcation from a center]\label{Teo:ChristopherCenter}
Consider system \eqref{eq:traceHopf}. For $\tau=0$, suppose that $\mu^\ast$ is a point in the Bautin variety such that the first $k$ Lyapunov coefficients have independent linear parts (with respect to the expansion of $L_i$ about $\mu^\ast$) and the next $l$ ones have their linear parts as linear combinations of the linear parts of the first $k$ ones. After a change of variables in the parameter space, if necessary, we can write $L_i=u_i$ for $i=1,\dots,k$ and, assuming $L_1=\dots=L_k=0$, the next Lyapunov coefficients write as $L_i=h_i(u)+O(|u|^{m+1})$, for $i=k+1,\dots,k+l$, where $h_i$ are homogeneous polynomials of degree $m\geqslant 2$ and $u=(u_{k+1},\dots,u_{k+l})$. If there exists a line $\eta$ in the parameter space such that for $k<i<k+l$, $h_i(\eta)=0$, the hypersurfaces $h_i=0$ intersect transversally along $\eta$, and $h_{k+l}(\eta)\neq 0$, then there exist parameter values $(\tau,\mu)$ for which system \eqref{eq:traceHopf} has $k+l$ limit cycles.
\end{teo}

The main difficulties of applying the above results when tackling the cyclicity problem is the computation of the Lyapunov coefficients whose expressions become substantially larger as we increase the number of coefficients. As a consequence, finding a suitable parameter value $\mu^\ast$ also becomes increasingly challenging. In this direction, in order to benefit from the whole of the computational power, one can obtain a numerical approximation of such a parameter value and then prove its existence analytically via the Poincaré--Miranda Theorem \cite{Kulpa97} along with the Gerschgorin Circles Theorem \cite{Golub13}.

\begin{teo}[Poincaré--Miranda]\label{Teo:PM} Let $c$ be a positive real number and $S=[-c,c]^n$ a $n$-dimensional cube. Consider $f=(f_1,\dots,f_n)$, $f_i:S\to\mathbb{R}$ a continuous mapping such that $f_i(S_i^+)\cdot f_i(S_i^-)<0$. for each $i\leqslant n$, where $S_i^\pm=\{(x_1,\dots,x_n)\in S:x_i=\pm c\}$. Then, there exists $d\in S$ such that $f(d)=0$.
\end{teo}

\begin{teo}[Gerschgorin Circles Theorem]\label{Teo:Gerschgorin}
Let $A=(a_{i,j})\in\mathbb{C}^{n\times n}$. For each $i\in\{1,\dots,n\}$, consider
$$D_i=\{z\in\mathbb{C}:|z-a_{i,i}|<R\},$$
where $R=\sum_{i\neq j}|a_{i,j}|$. If $\lambda$ is an eigenvalue of $A$, then there exists $i\in\{1,\dots,n\}$ such that $\lambda\in D_i$.
\end{teo}

The computer assisted approach was successfully implemented to obtain high lower bounds on the cyclicity of monodromic singularities in \cite{GouveiaGine, GouveiaTorre21}.

\section{First order analysis of monodromic singular points}\label{Sec:method}

We consider the perturbation \eqref{eq:01} of $X_0$, and search for conditions on the parameters $\mu\in\Lambda$ for invariant tori bifurcations to occur. To fix notation, we set $\mu=(a_{jkl},b_{jkl},c_{jkl})$ and 
$$X_1(\mu)=\sum_{j+k+l=0}^{m}a_{jkl} x^j y^k z^l\partial_x+\sum_{j+k+l=0}^{m}b_{jkl} x^j y^k z^l\partial_y+\sum_{j+k+l=0}^{m}c_{jkl} x^j y^k z^l\partial_z.$$

Consider the family of differential systems \eqref{eq:01}. We introduce the change of variables $$x=r{\rm Cs}\theta,\quad y=r^n{\rm Sn}\theta,\quad z=w,$$ where ${\rm Cs\theta}$ and ${\rm Sn\theta}$ are the \emph{generalized trigonometric functions}. More precisely, these are the solutions of the Cauchy problem $u'=-v$, $v'=u^{2n-1}$, $u(0)=1, v(0)=0$. The generalized trigonometric functions are $T$-periodic with period $T=2\sqrt{\frac{\pi}{n}}\frac{\Gamma(1/2n)}{\Gamma((n+1)/2n)}$. Note that for $n=1$, they define the usual cosine and sine functions.

We obtain the differential system:
\begin{small}
\begin{eqnarray}\label{eq:polar}
\dot{r}&=&\frac{\varepsilon}{r^{2n-1}}\left(\sum_{j+k+l=0}^{m}a_{jkl}r^{2n-1+j+nk}{\rm Sn}^k\theta{\rm Cs}^{2n-1+j}\theta w^l+b_{jkl}r^{n+j+nk}{\rm Sn}^{k+1}\theta{\rm Cs}^j\theta w^l\right)+O(\varepsilon^2),\nonumber\\
\dot{\theta}&=&r^{n-1}+\frac{\varepsilon}{r^{n+1}}\left(\sum_{j+k+l=0}^{m}b_{jkl}r^{j+kn+1}{\rm Sn}^{k}\theta{\rm Cs}^{j+1}\theta w^l-na_{jkl}r^{n+j+nk}{\rm Sn}^{k+1}\theta{\rm Cs}^{j}\theta w^l\right)+O(\varepsilon^2),\nonumber\\
\dot{w}&=&\varepsilon\sum_{j+k+l=0}^{m}c_{jkl}r^{j+nk}{\rm Sn}^{k}\theta{\rm Cs}^j\theta w^l+O(\varepsilon^2),
\end{eqnarray}
\end{small}
Notice that $\dot{\theta}$ does not vanish at $\varepsilon=0$ for $r\neq 0$. Thus, setting $\theta$ as the independent variable, we obtain the differential system $$\left(\frac{dr}{d\theta},\frac{dw}{d\theta}\right)=\varepsilon F_1(r,\theta,w;\mu)+\varepsilon^2\tilde{F}(r,\theta,w;\mu,\varepsilon),$$
which is written in the standard form \eqref{eq:Averaging}. Now we are set to apply the averaging method to study the bifurcation of invariant tori in the original system \eqref{eq:01}. We now compute explicitly the first averaged function $\mathbf{g}_1(r,w;\mu)$.
$$\mathbf{g}_1(r,w;\mu)=(\mathbf{g}_1^1(r,w;\mu),\mathbf{g}_1^2(r,w;\mu)),$$
where
\begin{small}
\begin{eqnarray}
\mathbf{g}_1^1(r,w;\mu)&=&\frac{1}{r^{3n-2}}\sum_{j+k+l=0}^{m}\left(a_{jkl}r^{2n-1+j+nk}w^l\int_{0}^{T}{\rm Sn}^k\theta{\rm Cs}^{2n-1+j}\theta d\theta\right.\nonumber\\
&&\left.+b_{jkl}r^{n+j+nk}w^l\int_{0}^{T}{\rm Sn}^{k+1}\theta{\rm Cs}^j\theta d\theta\right),\nonumber\\
\mathbf{g}_1^2(r,w;\mu)&=&\frac{1}{r^{n-1}}\sum_{j+k+l=0}^{m}c_{jkl}r^{j+nk}w^l\int_{0}^{T}{\rm Sn}^{k}\theta{\rm Cs}^j\theta d\theta.\nonumber
\end{eqnarray}    
\end{small}

We now state a useful result regarding the generalized trigonometric functions. See \cite{LyapunovBook} for a proof.
\begin{propo}\label{propo:trig}
For $p,q$ fixed natural numbers, consider the following quantity
$$I(p,q)=\int_{0}^{T}{\rm Sn}^p\theta{\rm Cs}^q\theta d\theta.$$
We have that $I(p,q)=0$ if either $p$ or $q$ are odd and $I(p,q)\neq 0$ when both $p$ and $q$ are even.
\end{propo}
Using Proposition \ref{propo:trig}, applying the following rescalings in the independent variable and in the parameter space
$$\theta=r^{n-1}\tilde\theta,\quad \tilde{a}_{jkl}=I(k,2n-1+j)\,a_{jkl},\quad \tilde{b}_{jkl}=I(k+1,j)\,b_{jkl},\quad \tilde{c}_{jkl}=I(k,j)\,c_{jkl},$$
and dropping the tildes, yields that the guiding system $(r',w')=\mathbf{g}_1(r,w;\mu)$ falls into the following family:
\begin{eqnarray}\label{eq:guidingformula}
r'&=&\sum_{j+k+l=1}^{m}a_{jkl}r^{j+nk}w^l+r^{1-n}\sum_{j+k+l=1}^{m}b_{jkl}r^{j+nk}w^l,\nonumber\\
w'&=&\sum_{j+k+l=0}^{m}c_{jkl}r^{j+nk}w^l,
\end{eqnarray}
for the following admissible parameters
$$a_{jkl}: j+1\equiv k\equiv 0 \mod 2,\quad b_{jkl}: j\equiv k+1\equiv 0 \mod 2,\quad c_{jkl}: j\equiv k\equiv 0 \mod 2.$$

Once this process is completed, one may search for equilibrium points $(\rho,w_0)$ of \eqref{eq:guidingformula} in the region $r>0$ whose associated eigenvalues are conjugated complex with non zero imaginary part, which means that the Jordan canonical form is given by \eqref{eq:traceHopf}. In practice, it is more effective to assume conditions on the parameters $(a_{jkl},b_{jkl},c_{jkl})$ such that $(\rho,w_0)$ is such a point. More precisely, we introduce the change of parameters:

\begin{equation}
\begin{aligned}
b_{010}&=-a_{100}-\sum_{j+k+l=2}^{m}\biggl(a_{jkl}\rho^{j+nk-1}w_0^l+b_{jkl}\rho^{j+n(k-1)}w_0^l\biggr),\\
c_{000}&=-\sum_{j+k+l=1}^{m}c_{jkl}\rho^{j+nk}w_0^l,
\end{aligned}
\end{equation}
which is well defined for $\rho>0$. Now, by setting
\begin{small}
\begin{equation}
\begin{aligned}
c_{001}&=2\tau-\sum_{j+k+l=2}^{m}\biggl((j+nk-1)a_{jkl}\rho^{j+nk-1}w_0^l+(j+n(k-1))b_{jkl}\rho^{j+n(k-1)}w_0^l+lc_{jkl}\rho^{j+nk}w_0^{l-1}\biggr),
\end{aligned}
\end{equation}
\end{small}
and assuming the condition
\begin{equation}\label{eq:conditionHopf}
d:=\sum_{j+k+l=2}^{m}l(a_{jkl}\rho^{j+nk}+b_{jkl}\rho^{j+n(k-1)+1})w_0^{l-1}\neq 0,
\end{equation}
we can introduce the additional change of parameter given by
\begin{equation}
\begin{aligned}
c_{200}&=(2\rho\, d)^{-1}\biggl(-\omega^2-\tau^2-2d\,n\,c_{020}\rho^{2n-1}-d\sum_{j+k+l=3}^{m}(j+nk)c_{jkl}\rho^{j+nk-1}w_0^l\\
&+\biggl(\sum_{j+k+l=2}^{m}\left[(j+nk-1)a_{jkl}\rho^{j+nk-1}+(j+n(k-1))b_{jkl}\rho^{j+n(k-1)}\right]w_0^l\biggr)\cdot\\
&\quad\quad\cdot\biggl(\sum_{j+k+l=1}^{m}l c_{jkl}\rho^{j+nk}w_0^{l-1}\biggr)\biggr),
\end{aligned}
\end{equation}
to obtain that the eigenvalues of $D\mathbf{g}_1(\rho,w_0)$ are given by $\tau\pm i\omega$.

The next step is to put system \eqref{eq:guidingformula} into its Jordan canonical form \eqref{eq:traceHopf} to investigate the limit cycles that unfold from a degenerate Hopf bifurcation at the singular point $(\rho,w_0)$. To this end, we apply  Theorem \ref{Teo:ChristopherFocus} or Theorem \ref{Teo:ChristopherCenter} to guarantee the existence of parameter values for which limit cycles unfold from $(\rho,w_0)$. For each limit cycle, we imply by Theorem \ref{Teo:InvariantTori}, the existence of a correspondent normally hyperbolic limit torus in the original system \eqref{eq:01}. In the next section we present the new lower bounds on $N_h(m)$ for $m=2,3,4,5$ that we have obtained by applying the method above.

\begin{obs}\label{obs:degreeguiding}
The guiding system \eqref{eq:guidingformula} is polynomial and one can easily determine its degree $d(m,n)$ in terms of the degree $m$ of the original system \eqref{eq:01} and the Andreev number $n$. More precisely, if $m$ is odd $d(m,n)=n(m-1)+1$ and if $m$ is even, $d(m,n)=nm$. Thus, the number of limit tori one can detect via the first ordered averaged system is bounded by $H(d(m,n))$, which suggests that in the search for better lower bounds on $N(m)$, the nilpotent-zero singular points have more potential. For instance, for quintic systems \eqref{eq:01}, the first order averaged function $\mathbf{g}_1$ has also degree $5$, but for $n=2,3$, the corresponding degrees are $9$ and $13$. 
\end{obs}

It is important to remark that the limit tori detected from the method described above are all nested and surround the same periodic solution of \eqref{eq:01}. This is due to the fact that these limit tori correspond to limit cycles arising from a degenerate Hopf bifurcation at the singular point $(\rho,w_0)$ of the guiding system. Since these limit cycles are nested, so are the limit tori of the original system.

\section{Lower bounds for $N_h(m)$}\label{sec:proof}

In this section we present the new lower bounds on the number of limit tori $N(m)$ for polynomial vector fields \eqref{eq:01} of degree $m$.

\subsection{Quadratic systems}

For quadratic systems \eqref{eq:01}, applying the method described in the previous section, after translating $(\rho,w_0)$ to the origin, we obtain the guiding system \eqref{eq:guidingformula} given by
\begin{equation}
\begin{aligned}
r'&=-\frac{d w_0}{\rho }r+\frac{d}{\rho}r w,\\
w'&=\frac{\rho  \left(\tau ^2+\omega ^2\right)}{2 d}+c_{002} w_{0}^2-2 \tau  w_{0}+(2 \tau -2 c_{002} w_{0})w-\frac{\tau ^2+\omega ^2}{2 d \rho }r^2+c_{002}w^2,
\end{aligned}
\end{equation}
whose Jordan canonical form, after translating $(\rho,w_0)$ for $\tau=0$, is given by
\begin{equation}
\begin{aligned}
x'&=-\omega y+\omega c_{002} x^2-\frac{d}{2\rho}y^2,\\
y'&=\omega x+\frac{d \omega}{\rho}xy,
\end{aligned}
\end{equation}
which has a time-reversible center at the origin. Thus, the equilibrium does not undergo a Hopf bifurcation at $\tau=0$. We conclude that, using only the first order averaging functions, we do not detect any invariant torus for system \eqref{eq:01} and therefore, a higher order analysis is required to obtain a new lower bound for $N(2)$. This is precisely what is done in Section \ref{Sec:N2>3}.

\subsection{Cubic systems}\label{subsec:m3}
For cubic systems \eqref{eq:01}, there are two different forms for $X_0$, namely $n=1,2$, which correspond respectively to the Hopf-zero and the nilpotent-zero equilibria. In this section, we present a comparison between the two cases and show that indeed one can detect more limit tori unfolding from the nilpotent-zero equilibria than the Hopf-zero ones, as discussed in Remark \ref{obs:degreeguiding}. 

\subsubsection*{Hopf-zero equilibrium}

Applying the method described in Section \ref{Sec:method}, after translating the equilibrium $(\rho,w_0)$ to the origin and rescaling the independent variable by $\omega$, the Jordan canonical form of the guiding system \eqref{eq:guidingformula} is given by the following 5-parameter family, for $\tau=0$:
\begin{eqnarray}\label{eq:ZHcubicJordan}
x'&=&-y+\frac{2 \lambda_1^3 \rho}{\omega}x y^2-\frac{\lambda_1^2 \lambda_3 \rho}{\omega }x y^2+\frac{\lambda_1^2 \lambda_4 \rho }{\omega} y^3+2 \lambda_1^2 x^2 y+\frac{2 \lambda_1^2 \rho }{\omega }x y+\lambda_1^2 y^3+\lambda_1 \lambda_2 y^3\nonumber\\
&&-\frac{1}{2} \lambda_1 \lambda_3 x^2 y-\frac{\lambda_1 \lambda_3 \rho }{\omega }x y+\lambda_1 \lambda_4 x y^2-\frac{\lambda_1 \lambda_5 \rho }{\omega }y^3+\frac{\lambda_1 \omega }{2 \rho }x^3 +\lambda_1 x^2+\frac{\lambda_1  \omega }{\rho }x y^2\nonumber\\
&&+\frac{\lambda_2 \omega }{\rho }x y^2+\lambda_2 y^2-\frac{\omega }{\rho }x y ,\nonumber\\
y'&=&x-\lambda_1^2 x y^2+\lambda_1 \lambda_3 x y^2-\frac{\lambda_1 \omega}{2 \rho }x^2 y+\frac{\lambda_3  \omega}{2 \rho
   }x^2y+\lambda_3 x y+\lambda_4 y^2+\lambda_5 y^3+\frac{ \omega }{2 \rho } x^2.
\end{eqnarray}
Hence, any lower bound for the cyclicity of the origin of the family \eqref{eq:ZHcubicJordan} implies by Theorem \ref{Teo:InvariantTori} in a lower bound for the number of limit tori for the original cubic system \eqref{eq:01} with $n=1$. We compute the Lyapunov coefficients $L_i(\mu)$ for $\mu=(\lambda_1,\lambda_2,\lambda_3,\lambda_4,\lambda_5)$, $i=1,\dots,6$ and obtain:
\begin{eqnarray}
L_1(\mu)&=&\frac{\left(2 \lambda_4 \lambda_1 +4 \lambda_2 \lambda_4 -2 \lambda_4 \lambda_3 +6 \lambda_5 \right) \omega +2 \left(2\lambda_1 +\lambda_2\right) \rho  \left(2\lambda_1 -\lambda_3\right) \lambda_1}{3 \omega},\nonumber\\
L_2(\mu)&=&\frac{2}{45 \omega^{3}}\left(10 \lambda_1 +5 \lambda_2 -\lambda_3 \right)L_{2,1}(\mu)L_{2,2}(\mu)\mod\langle L_1\rangle,\nonumber\\
L_3(\mu)&=&\frac{2 \lambda_3}{1575 \omega^{3}} \left(\lambda_2 +\lambda_3 \right) \left(5 \lambda_2 +\lambda_3 \right)L_{2,1}(\mu)L_{2,2}(\mu)\mod\langle10\lambda_1+5\lambda_2-\lambda_3,\,L_1\rangle,\nonumber
\end{eqnarray}
with
\begin{eqnarray}
L_{2,1}(\mu)&=& 2 \lambda_1^{3} \rho^{2}-\lambda_1^{2} \rho^{2} \lambda_3 +2 \lambda_1 \rho  \omega  \lambda_4 +\lambda_1 \,\omega^{2}+\omega^{2} \lambda_3,\nonumber\\
L_{2,2}(\mu)&=&4 \lambda_1^{3} \rho +2 \lambda_1^{2} \lambda_2 \rho -2 \lambda_1^{2} \lambda_3 \rho -\lambda_1 \lambda_2 \lambda_3 \rho -2 \lambda_1 \lambda_4 \omega +2 \lambda_2 \lambda_4 \omega +2 \lambda_3 \lambda_4 \omega.
\end{eqnarray}
For parameter values $\mu^\ast$ in the variety $\{L_1(\mu)=0\}\cap\{10\lambda_1+5\lambda_2-\lambda_3=0\}$, we have
$${\rm rank}\left[\dfrac{\partial(L_1,L_2)}{\partial\mu}\right]_{\mu=\mu^\ast}=2.$$
Thus, by Theorem \ref{Teo:ChristopherFocus}, there exist parameter values $\mu^\ast$ such that system \eqref{eq:ZHcubicJordan} has three small amplitude limit cycles surrounding the origin, which in turn, by Theorem \ref{Teo:InvariantTori} implies that, for this parameter value $\mu^\ast$ and $\varepsilon>0$ sufficiently small, \eqref{eq:01} has three normally hyperbolic limit tori.

One particular example that exhibits this behavior is the following system:

\begin{eqnarray}\label{eq:0Hopf3tori}
\dot{x}&=&-y+\varepsilon \frac{x}{15} \left( 15\delta{z}^{2}-30{\tau}^{2}z+40\tau{x}^{2}+60\tau z-20{x}^{2}+24{z}^{2}-30\tau-60z+15 \right)\nonumber\\
\dot{y}&=&x,\\
\dot{z}&=&\frac{\varepsilon}{30}\left( 30{x}^{2}-{z}^{3} \left( 16-15\beta+
10\delta \right) +30 z -15 \right)
.\nonumber
\end{eqnarray}
In a neighborhood of $(\tau,\beta,\delta)=(0,0,0)$, there exist a point of parameter values $(\tau^\ast,\beta^\ast,\delta^\ast)$ such that system \eqref{eq:0Hopf3tori} has three invariant tori in its phase space. Numeric simulations of the solutions of \eqref{eq:0Hopf3tori} suggest that $(\tau^\ast,\beta^\ast,\delta^\ast)=(-6\times 10^{-6},6\times 10^{-3},6\times10^{-2})$ is one such point.

\subsubsection*{Nilpotent-zero equilibrium}

As we have anticipated in Remark \ref{obs:degreeguiding}, we were able to detect more limit tori for the nilpotent-zero case, i.e. $n=2$, than in the Hopf-zero case. In the next result, we obtain a new lower bound for $N(3)$ by exhibiting a particular family of cubic systems for which five limit tori unfold in a neighborhood of the nilpotent-zero singularity at the origin. 

\begin{teo}\label{N(3)=5}
Consider the following cubic family of differential systems
\begin{equation}\label{eq:ZNilcubic}
\begin{aligned}
\dot{x}&=-y+\varepsilon\biggl(\lambda_{1}x+4 x z+7 \lambda_{1} x y^{2} +4 x z^{2} \lambda_{2}-\frac{5 \pi^{2}\lambda_{1} }{9 \Gamma \! \left(\frac{3}{4}\right)^{4}}x^{3}\biggr),\\
\dot{y}&=x^3,\\
\dot{z}&=\varepsilon\biggl(\frac{ \left(16\tau-5 \lambda_{4}\right)}{6}z+\frac{\pi^{2} \left(2 \lambda_{3}-2 \tau^{2} -3\right)}{6 \Gamma \! \left(\frac{3}{4}\right)^{4}}x^{2} +\frac{ \left(1-2 \lambda_{3} \right)}{2}y^{2}+\frac{\pi^{2}\lambda_{4}}{2 \Gamma \! \left(\frac{3}{4}\right)^{4}}x^{2} z-\frac{\lambda_{4}}{2}y^{2} z
\biggr).
\end{aligned}
\end{equation}
There exist parameter values for $\tau,\lambda_1,\lambda_2,\lambda_3,\lambda_4$ such that for sufficiently small $\varepsilon>0$, there are five limit tori in its phase space.
\end{teo}
\begin{proof}
After applying the method described in Section \ref{Sec:method}, we obtain, after a constant rescaling, the corresponding guiding system \eqref{eq:guidingformula} given by
\[
\begin{aligned}
r'&=\frac{1}{4} \lambda_{1} r \left(r^2-1\right)^2+r w (\lambda_{2} w+1),\\
w'&=\frac{1}{8} \left(16\tau w-2 \lambda_{3} \left(r^2-1\right)^2-\left(r^4-6
   r^2+5\right) (\lambda_{4} w-1)-4 \left(r^2-1\right) \tau ^2\right),
\end{aligned}
\]
for which the singular point $(\rho,w0)=(1,0)$ has the associated eigenvalues $\tau\pm i$. We then put the guiding system into its Jordan canonical form which, for $\tau=0$, is given by
\begin{equation}\label{eq:ZNilcubicJordan}
\begin{aligned}
x'&=-y+\frac{1}{4} \lambda_{1} x^2 (x+1) (x+2)^2+y (\lambda_{2} (x+1) y-x),\\
y'&=x+\frac{1}{8} x \left(2 \lambda_{3} (x+2)^2 x-\left((x+4) x^2 (\lambda_{4}
   y+1)\right)+8 \lambda_{4} y\right).
\end{aligned}
\end{equation}
We then compute the Lyapunov coefficients $L_1(\mu),\dots,L_5(\mu)$ for $\mu=(\lambda_1,\lambda_2,\lambda_3,\lambda_4)$. Due to the size of their expressions, we present only the first two of them.
\[
\begin{aligned}
L_1(\mu)&=\frac{10}{3} \lambda_{1} -\frac{4}{3} \lambda_{1} \lambda_{3} -\frac{2}{3} \lambda_{3} \lambda_{4},\\
L_2(\mu)&=-\frac{19}{6} \lambda_{1} +\frac{19}{30} \lambda_{3} \lambda_{4} +\frac{407}{45} \lambda_{1} \lambda_{3}+\frac{2}{3} \lambda_{1}^{2} \lambda_{4} -\frac{4}{3} \lambda_{1}^{2} \lambda_{2} -\frac{14}{9} \lambda_{3}^{2} \lambda_{4} -\frac{98}{9} \lambda_{1}\lambda_{3}^{2}-\frac{16}{3} \lambda_{1}^{3}+\frac{14}{9} \lambda_{3}^{3} \lambda_{4}\\
&\quad +\frac{8}{15} \lambda_{4} \lambda_{1} \lambda_{2} \lambda_{3} +\frac{32}{15} \lambda_{1}^{3} \lambda_{3} +\frac{28}{9} \lambda_{1} \lambda_{3}^{3}-\frac{2}{15} \lambda_{3} \lambda_{4}^{3}+\frac{8}{5} \lambda_{1}^{2} \lambda_{3} \lambda_{4} +\frac{8}{15} \lambda_{1}^{2} \lambda_{2} \lambda_{3} +\frac{2}{15} \lambda_{2} \lambda_{3} \lambda_{4}^{2}.\\
\end{aligned}
\]
The result will follow, by Theorem \ref{Teo:ChristopherCenter}, from the existence of a parameter value $\mu^\ast$ such that $L_1(\mu^\ast)=L_2(\mu^\ast)=L_3(\mu^\ast)=L_4(\mu^\ast)=0$, $L_5(\mu^\ast)\neq 0$ and
$${\rm rank}\left[\frac{\partial(L_1,\dots,L_4)}{\partial\mu}\right]_{\mu=\mu^\ast}=4.$$
A numerical approximation for such a parameter is given by 
$$\hat{\mu}=(-0.1121033429,1.640345056,-0.5190794622,1.304035026),$$
for which $L_5(\hat{\mu})=-0.7309104115$ and $\det\left[\frac{\partial (L_1,\dots,L_4)}{\partial\mu}\right]_{\mu=\hat{\mu}}=-0.0894307229$. We have obtained this approximation by solving numerically $L_i=0$ with different number of digits up to 1000, to ensure the convergence of the numerical solution. The above numerical approximation is shown with only 10 digits.

Now, we perform a computer assisted proof to obtain analytically the existence of $\mu^\ast$. This is done using the Poincaré--Miranda Theorem. In order to do so, we first convert the approximate solution to the adequate rational expressions:
\begin{eqnarray}
\hat{\lambda}_1&=-\dfrac{11210334284144350283}{100000000000000000000},\quad \hat{\lambda}_2&=\dfrac{16403450551921741039}{10000000000000000000},\nonumber\\
\hat{\lambda}_3&=-\dfrac{25953973096128722529}{50000000000000000000},\quad \hat{\lambda}_4&=\dfrac{6520175133339217977}{5000000000000000000}.\nonumber
\end{eqnarray}
Next, we introduce an affine parameter change which translates $\hat{\mu}$ to the origin and such that the Jacobian matrix of $(L_1,\dots,L_4)$ at $\mu=\hat{\mu}$ is near the identity. Then, we apply the Theorem \ref{Teo:PM}, with $S=[-10^{-10},10^{-10}]^4$ and $f=(L_1,L_2,L_3,L_4)$. Performing an accurate interval analysis (see, for instance, \cite{GouveiaTorre21}), we obtain that
$$f_i(S_i^{+})\in[9.999999355\times10^{-11},1.000000079\times10^{-10}],\,\text{and}$$
$$f_i(S_i^{-})\in[ -1.000000065\times10^{-10}, -9.999999214\times10^{-11}].$$
Thus, by Theorem \ref{Teo:PM}, there exists $\mu^\ast$ such that $f(\mu^\ast)=0$. Moreover, we have $L_5(S)\in[0.9999999548,1.000000046]$ and thus $L_5(\mu^\ast)\neq 0$.

Finally, to finish the proof, we apply Theorem \ref{Teo:Gerschgorin} to conclude that $\det Jf(\mu^\ast)\neq 0$. Denoting by $Jf(\mu^\ast)=(\alpha_{i,j})$, we have that $\sum_{i\neq j}|\alpha_{i,j}|< 1.156775391\times 10^{-16}$ and $\alpha_{i,i}\in[0.9999999999,1.0000000001]$. Therefore, by the Gerschgorin Circles Theorem, all the eigenvalues of $Jf(\mu^\ast)$ are non-zero and thus, $\det Jf(\mu^\ast)\neq 0$. Hence, there exist parameter values $(\tau,\mu)$ near $(0,\mu^\ast)$ such that the guiding system has five limit cycles, which in turn, by Theorem \ref{Teo:InvariantTori} implies that system \eqref{eq:ZNilcubic} has five limit tori.
\end{proof}

From the above result, obtain the new lower bound for the number of limit tori in the phase space of cubic three-dimensional vector fields: $N_h(3)\geqslant 5$.

\subsection{Quartic systems}\label{subsec:m4}
We perform the same procedure as in the previous section to study the systems \eqref{eq:01} of degree $m=4$. Once more, the number of limit tori we were able to detect using the first order averaging method was higher for the nilpotent-zero singular points than for the Hopf-zero ones. More precisely, we were able to obtain examples of quartic systems with a Hopf-zero equilibrium point from which six limit tori can bifurcate. As for the quartic systems with nilpotent-zero equilibria, we were able to obtain examples with seven limit tori. The next result presents the latter, from which we consequentially obtain the new lower bound for $N(4)$, namely $N_h(4)\geqslant 7$.

\begin{teo} \label{N(4)=7} Consider the following quartic family of differential systems
\begin{equation}\label{eq:ZNilquartic}
\begin{aligned}
\dot{x}&=-y+\varepsilon\bigg(\frac{\lambda_{1} x}{4}-\frac{\left(7\lambda_{12}-20 -9 \lambda_{6} \right)}{20} x z-\frac{5 \lambda_{1} \pi^{2}}{36 \Gamma \! \left(\frac{3}{4}\right)^{4}}x^{3}+\frac{7 \lambda_{1} }{4}x y^{2}+\lambda_{2} x z^{2}\\
&\quad +\frac{5 \pi^{2} \left(\lambda_{12} -\lambda_{6} \right) }{36 \Gamma \! \left(\frac{3}{4}\right)^{4}}x^{3} z-\frac{7 \left(3 \lambda_{12} -\lambda_{6} \right)}{20} x y^{2} z\bigg),\\
\dot{y}&=x^3,\\
\dot{z}&=\varepsilon\bigg(\frac{\tau^{2}}{6}+\frac{31}{128}+\frac{\left(2 \lambda_{9} -5 \lambda_{4} +16 \tau \right) }{24}z-\frac{\pi^{2} \left(16 \tau^{2}+35\right) }{192 \Gamma \! \left(\frac{3}{4}\right)^{4}}x^{2}+\frac{35}{64}y^{2}+\frac{5 \lambda_{10}}{24}z^{2}+\frac{\lambda_{11}}{3}z^{3}\\
&\quad +\frac{\pi^{2} \left(3 \lambda_{4} -2 \lambda_{9} \right) }{24 \Gamma \! \left(\frac{3}{4}\right)^{4}}x^{2} z +\frac{\left(2 \lambda_{9} -\lambda_{4} \right)}{8} y^{2} z+\frac{35}{384}y^{4} +\frac{\lambda_{10} }{8}y^{2} z^{2}-\frac{\lambda_{15} }{3}z^{4}-\frac{35 \pi^{2} }{192 \Gamma \! \left(\frac{3}{4}\right)^{4}}x^{2} y^{2}\\
&\quad -\frac{\pi^2\lambda_{10} }{8 \Gamma \! \left(\frac{3}{4}\right)^{4}}x^{2} z^{2}
\bigg).
\end{aligned}
\end{equation}
There exist parameter values for $\tau,\lambda_1,\lambda_{2},\lambda_{4},\lambda_{6},\lambda_{9},\lambda_{10},\lambda_{11},\lambda_{12},\lambda_{15}$ such that for sufficiently small $\varepsilon>0$, there exist seven limit tori in its phase space.
\end{teo}
\begin{proof}
The guiding system \eqref{eq:guidingformula} corresponding to system \eqref{eq:ZNilquartic} is given by
\begin{small}
\begin{equation}\label{eq:ZNilquarticJordan}
\begin{aligned}
r'&=\frac{1}{20} r \left(5 \lambda_{1} \left(r^2-1\right)^2+\left(r^4 (\lambda_{6}-3 \lambda_{12})+10 r^2 (\lambda_{12}-\lambda_{6})-7 \lambda_{12}+20
   \lambda_{2} w+9 \lambda_{6}+20\right)w\right),\\
w'&=\lambda_{11} w^3-\lambda_{15} w^4+\frac{1}{8} w \left(-\lambda_{4} \left(r^4-6 r^2+5\right)+2
   \lambda_{9} \left(r^2-1\right)^2+16 \tau \right)\\
   &\quad +\frac{1}{8} \lambda_{10} \left(r^4-6 r^2+5\right) w^2+\frac{1}{128} \left(r^2-1\right) \left(5 r^6-23 r^4+47 r^2-64 \tau ^2-93\right).
\end{aligned}
\end{equation}    
\end{small}
We have that $(\rho,w0)=(1,0)$ is a singular point of the above system, whose associated eigenvalues are $\tau\pm i$. For $\tau=0$, the Jordan canonical form of the guiding system is given by
\begin{equation}
\begin{aligned}
x'&=-y+\frac{1}{20} \biggl(5 \lambda_{1} x^2 (x+1) (x+2)^2+y \biggl(20 \lambda_{6} x^2-4 x (2 \lambda_{12}-5 \lambda_{2} y-4 \lambda_{6}+5)\\
&\quad + (3 \lambda_{12}-\lambda_{6})x^5+5(3\lambda_{12}-\lambda_{6})x^4+20 \lambda_{12} x^3+20 \lambda_{2} y\biggr)\biggr),\\
y'&=x-\frac{1}{8} x^4 y (\lambda_{10} y+\lambda_{4}-2 \lambda_{9})-\frac{1}{2} x^3 y (\lambda_{10} y+\lambda_{4}-2 \lambda_{9})+x y (\lambda_{10} y+\lambda_{4})+\lambda_{9} x^2 y\\
&\quad+y^3
   (\lambda_{11}+\lambda_{15} y)-\frac{5 x^8}{128}-\frac{5 x^7}{16}-\frac{7 x^6}{8}-\frac{7 x^5}{8}.
\end{aligned}
\end{equation}
 We then compute the Lyapunov coefficients $L_1(\mu),\dots,L_7(\mu)$, with 
 \[\mu=(\lambda_1,\lambda_{2},\lambda_{4},\lambda_{6},\lambda_{9},\lambda_{10},\lambda_{11},\lambda_{12},\lambda_{15}).\]
 Notice that for $\mu=0$, the above system is time-reversible and therefore, $\mu=0$ correspond to a point in the Bautin variety $V^c$. Denoting by $L_i^k(\mu)$ the homogeneous term of order $k$ in the power series expansion of $L_i$ about $\mu=0$, we have that
\[
\begin{aligned}
L_1^1(\mu)&=\frac{10 \lambda_{1}}{3}+2 \lambda_{11} +\frac{2 \lambda_{9}}{3},\\
L_2^1(\mu)&=-\frac{58 \lambda_{1}}{15}-\frac{16 \lambda_{11}}{5}-\frac{11 \lambda_{4}}{60}-\frac{7 \lambda_{9}}{10},\\
L_3^1(\mu)&=\frac{1511 \lambda_{1}}{140}+\frac{1481 \lambda_{11}}{210}+\frac{181 \lambda_{4}}{180}+\frac{2213 \lambda_{9}}{1260}.
\end{aligned}
\]
After the change of variables $L_i=u_i$, for $i=1,2,3$, we have that, for $u_1=u_2=u_3=0$,
\[
\begin{aligned}
L_4^2(\mu)&=-\tfrac{8838413}{7068600} \lambda_{12} \lambda_{2} +\tfrac{1492013}{7068600} \lambda_{6} \lambda_{2} -\tfrac{3293837}{353430} \lambda_{15} \lambda_{2} +\tfrac{22554373}{20790000} \lambda_{12} \lambda_{9} +\tfrac{777737}{2970000} \lambda_{9} \lambda_{6}\\
&\quad +\tfrac{777737}{1386000} \lambda_{10} \lambda_{9} +\tfrac{8887}{47600} \lambda_{10} \lambda_{2} +\tfrac{1000259}{841500} \lambda_{15} \lambda_{9},\\
L_5^2(\mu)&=\tfrac{5185872127}{933055200} \lambda_{12} \lambda_{2} -\tfrac{129360589}{133293600} \lambda_{6} \lambda_{2} +\tfrac{939872357}{23326380} \lambda_{15} \lambda_{2} -\tfrac{955826051}{171517500} \lambda_{12} \lambda_{9} -\tfrac{32959519}{24502500} \lambda_{9} \lambda_{6}\\
&\quad -\tfrac{32959519}{11434500} \lambda_{10} \lambda_{9} -\tfrac{6236687}{7068600} \lambda_{10} \lambda_{2} -\tfrac{2057892563}{388773000} \lambda_{15} \lambda_{9},\\
L_6^2(\mu)&=\tfrac{3145321040171}{509448139200} \lambda_{6} \lambda_{2} -\tfrac{80101673692}{318405087} \lambda_{15} \lambda_{2} +\tfrac{2149056494527}{59935075200} \lambda_{12} \lambda_{9} +\tfrac{74105396363}{8562153600} \lambda_{9} \lambda_{6}\\
&\quad +\tfrac{74105396363}{3995671680} \lambda_{10} \lambda_{9} +\tfrac{1577199913}{280687680} \lambda_{10} \lambda_{2} +\tfrac{713875167457}{21227005800} \lambda_{15} \lambda_{9} -\tfrac{17988230785463}{509448139200} \lambda_{12} \lambda_{2}.
\end{aligned}
\]
Now, for parameter values
\[
\begin{aligned}
\lambda_{10} &=-\frac{4417345784047540848}{126350370961756205} \lambda_{15} \alpha -8 \lambda_{12} -\frac{465380402374596362472}{631751854808781025} \lambda_{15},\\
\lambda_{2} &= \frac{\lambda_{9} \alpha}{4},\\ \lambda_{6} &= 
\frac{1472448594682513616}{25270074192351241} \lambda_{15} \alpha +13 \lambda_{12} +\frac{199448743874827012488}{126350370961756205} \lambda_{15},
\end{aligned}
\]
where
\[\alpha=-\frac{2693013801128322183\pm\sqrt{7224416669143576590544700439372137397}}{184056074335314202},\]
we have that $L_4^2=L_5^2=L_6^2=0$,
\[
L_7^2=-\frac{2897596542100597697456531 \lambda_{15} \lambda_{9} \alpha}{163913830847202477721680},\quad {\rm rank}\left[\dfrac{\partial(L_4^2,L_5^2,L_6^2)}{\partial(\lambda_{10},\lambda_{2},\lambda_{6})}\right]=3.
\]
Therefore, assuming additionally $\lambda_{9}\lambda_{15}\neq 0$, by Theorem \ref{Teo:ChristopherCenter}, there are parameter values $(\tau,\mu)$ for which the guiding system \eqref{eq:ZNilquarticJordan} has seven limit cycles and consequentially, system \eqref{eq:ZNilquartic} has seven limit torus in its phase space.
\end{proof}

\subsection{Quintic systems}\label{subsec:m5}

For quintic systems, there are two types of nilpotent-zero singularities, respectively, those with Andreev number $n=2,3$. As illustrated in Section \ref{Sec:method} and by the previous results, via the averaging method, one can detect more limit tori from the nilpotent-zero singularities then the Hopf-zero ones. However, there is also a distinction between the nilpotent-zero singularities in what concerns limit tori bifurcation in the sense that, the more degenerate the singularity is (i.e. the greater the value of $n$), the more limit tori can be obtained via the first order averaging method. The next two results, stated as Theorems \ref{Teo:Quintic11tori} and \ref{N(5)=13} exemplify this distinction. Since their proofs are the same, except for the specific functions and expressions, we only present the proof of Theorem \ref{N(5)=13}, for it is the result which gives the best lower bound for $N_h(5)$.  

\begin{teo}[Andreev number $n=2$]\label{Teo:Quintic11tori} Consider the following quintic family of differential systems
\begin{small}
\begin{equation}\label{eq:ZNiln2quintic}
\begin{aligned}
\dot{x}&=-y+\varepsilon\biggl(\frac{\left(35 \lambda_{1} -12 \lambda_{18} \right) }{128}x-\frac{\left(7 \lambda_{12}-20 \right)}{20} x z-\lambda_{19} x \,z^{2}+\frac{7 \left(21 \lambda_{1} -20 \lambda_{18} \right) }{64}x y^{2}+\frac{5\pi^{2} \lambda_{12} }{36 \Gamma \! \left(\frac{3}{4}\right)^{4}} x^{3} z\\
&\quad -\frac{21 \lambda_{12}}{20} x y^{2} z-\frac{5 \pi^{2} \left(\lambda_{1} -4 \lambda_{18} \right) }{64 \Gamma \! \left(\frac{3}{4}\right)^{4}}x^{3}y^{2}+\frac{77 \left(\lambda_{1} -4 \lambda_{18} \right)}{1152} x y^{4}+\lambda_{20} x z^{4}
\biggr),\\
\dot{y}&=x^3-\varepsilon\biggl(\frac{5 \pi^{2}\left(55 \lambda_{1} -28 \lambda_{18} \right)}{576 \Gamma \! \left(\frac{3}{4}\right)^{4}}x^{2} y -\frac{5 \pi^{2} \lambda_{19}}{6 \Gamma \! \left(\frac{3}{4}\right)^{4}} x^{2} y z^{2}
\biggr),\\
\dot{z}&=\varepsilon\biggl(\frac{\tau^{2}}{6}+\frac{7 \lambda_{8}}{96}+\frac{31}{128}+\frac{\left(29 \lambda_{9} +128 \tau -14 \lambda_{14} \right) }{192}z -\frac{\pi^{2} \left(16 \tau^{2}+24 \lambda_{8} +35\right) }{192 \Gamma \! \left(\frac{3}{4}\right)^{4}}x^{2}+\frac{5 \lambda_{10} }{24}z^{2}\\
&\quad +\frac{5 \left(7+12 \lambda_{8} \right)}{64} y^{2}+\frac{\pi^{2} \left(12 \lambda_{14} -19 \lambda_{9} \right) }{96 \Gamma \! \left(\frac{3}{4}\right)^{4}}x^{2} z+\frac{\left(8 \lambda_{11} -5 \lambda_{16}\right) }{24}z^{3}-\frac{5 \left(6 \lambda_{14} -7 \lambda_{9} \right) }{32}y^{2} z\\
&\quad -\frac{5 \pi^{2} \left(16 \lambda_{8} +7\right) }{192 \Gamma \! \left(\frac{3}{4}\right)^{4}}x^{2} y^{2}-\frac{\pi^{2}\lambda_{10} }{8 \Gamma \! \left(\frac{3}{4}\right)^{4}}x^{2} z^{2}+\frac{7 \left(5+12 \lambda_{8} \right)}{384} y^{4}-\frac{\lambda_{17}}{3}z^{4}+\frac{\lambda_{10}}{8} y^{2} z^{2}\\
&\quad -\frac{7 \left(6 \lambda_{14} -5 \lambda_{9} \right)}{192} y^{4} z-\frac{\lambda_{16} }{8}y^{2} z^{3}+\frac{\pi^{2}\lambda_{16} }{8 \Gamma \! \left(\frac{3}{4}\right)^{4}}x^{2} z^{3}+\frac{5 \pi^{2} \left(8 \lambda_{14} -7 \lambda_{9} \right) }{96 \Gamma \! \left(\frac{3}{4}\right)^{4}}x^{2} y^{2} z+\frac{\lambda_{23}}{3} z^{5}
\biggr).
\end{aligned}
\end{equation}
\end{small}
There exist parameter values for $\tau,\lambda_{1},  \lambda_{8}, \lambda_{9},\lambda_{10}, \lambda_{11}, \lambda_{12}, \lambda_{14}
, \lambda_{16}, \lambda_{17}, \lambda_{18}, \lambda_{19}, \lambda_{20}
, \lambda_{23}$ such that for sufficiently small $\varepsilon>0$, there exist eleven limit tori in its phase space.
\end{teo}

\begin{teo}[Andreev number $n=3$]\label{N(5)=13} Consider the following quintic family of differential systems

\begin{small}
\begin{equation}\label{eq:ZNiln3quintic}
\begin{aligned}
\dot{x}&=-y+\varepsilon\biggl(\frac{\left(4283 \lambda_{1}-1772 \lambda_{18} +2552 \lambda_{25} \right)}{3840}x+4 x z-\frac{\left(143 \lambda_{1}-572 \lambda_{18} +1112 \lambda_{25} \right)}{80} xy^{2}\\
&\quad -\frac{4 \left(35 \lambda_{19} -34 \lambda_{26} \right)}{35} x z^{2}+\frac{\pi^{3}2^{\frac{1}{3}} \sqrt{3} \left(9009 \lambda_{1} -10436 \lambda_{18} +17256 \lambda_{25} \right) }{54000 \Gamma \! \left(\frac{2}{3}\right)^{6}} x^{5}+\frac{8 \lambda_{26} }{7}x y^{2} z^{2}\\
&\quad +\frac{\pi 2^{\frac{2}{3}} \sqrt{3}   \left(429 \lambda_{1} -1716 \lambda_{18} +4936 \lambda_{25} \right) }{4800 \Gamma \! \left(\frac{2}{3}\right)^{3}}x^{3} y^{2}-\frac{\left(7\lambda_{1}-28 \lambda_{18} +88 \lambda_{25} \right)}{360}x y^{4}+4 \lambda_{20} x z^{4}
\biggr),\\
\dot{y}&=x^5-\varepsilon\biggl(\frac{\pi 2^{\frac{2}{3}} \sqrt{3} \left(5679 \lambda_{1} -3516 \lambda_{18} +5336 \lambda_{25} \right) }{2400 \Gamma \! \left(\frac{2}{3}\right)^{3}}x^{2} y-\frac{4 \pi  2^{\frac{2}{3}} \sqrt{3} \left(\lambda_{19} -\lambda_{26} \right)}{\Gamma \! \left(\frac{2}{3}\right)^{3}} x^{2} y z^{2}\biggr),\\
\dot{z}&=\varepsilon\biggl(\frac{\tau^{2}}{2}+\frac{1151}{1536}-\frac{\left(213 \lambda_{14} +20 \lambda_{27} -383 \lambda_{9} -1536 \tau \right)}{768}z -\frac{\pi 2^{\frac{2}{3}} \sqrt{3} \left(32 \tau^{2}+77\right) }{96 \Gamma \! \left(\frac{2}{3}\right)^{3}} x^{2}-\frac{77 }{48}y^{2}\\
&\quad+\frac{\pi 2^{\frac{2}{3}} \sqrt{3}  \left(33 \lambda_{14} +4 \lambda_{27} -45 \lambda_{9} \right) }{48 \Gamma \! \left(\frac{2}{3}\right)^{3}} x^{2} z+\frac{\left(93 \lambda_{14} -77 \lambda_{9} +20 \lambda_{27} \right) }{24}y^{2} z\\
&\quad +\frac{\left(2 \lambda_{23} -7 \lambda_{16} +12 \lambda_{11} \right) }{12}z^{3}+\frac{385 \pi^{3} 2^{\frac{1}{3}} \sqrt{3} }{1728 \Gamma \! \left(\frac{2}{3}\right)^{6}}x^{4}+\frac{55 \pi 2^{\frac{2}{3}} \sqrt{3}}{128 \Gamma \! \left(\frac{2}{3}\right)^{3}}x^{2} y^{2}-\frac{35}{576}y^{4}-\lambda_{17} z^{4}\\
&\quad -\frac{\pi  \left(69 \lambda_{14} +20 \lambda_{27} -55 \lambda_{9} \right) 2^{\frac{2}{3}} \sqrt{3}\, x^{2} y^{2} z}{64 \Gamma \! \left(\frac{2}{3}\right)^{3}} +\frac{5 \left(9 \lambda_{14} +4 \lambda_{27} -7 \lambda_{9} \right) }{288}y^{4} z-\frac{\left(\lambda_{16} -2 \lambda_{23} \right) }{6}y^{2} z^{3}\\
&\quad+\frac{\pi 2^{\frac{2}{3}} \sqrt{3} \left(5 \lambda_{16} -2 \lambda_{23} \right) }{12 \Gamma \! \left(\frac{2}{3}\right)^{3}} x^{2} z^{3}-\frac{5 \pi^{3} 2^{\frac{1}{3}} \sqrt{3} \left(75 \lambda_{14} +12 \lambda_{27} -77 \lambda_{9} \right) }{864 \Gamma \! \left(\frac{2}{3}\right)^{6}} x^{4} z+\lambda_{24} z^{5}
\biggr).
\end{aligned}
\end{equation}   
\end{small}
There exist parameter values for $\tau,\lambda_{1},\lambda_{9},\lambda_{11}, \lambda_{14}, \lambda_{16}, \lambda_{17}
, \lambda_{18}, \lambda_{19}, \lambda_{20}, \lambda_{23}, \lambda_{24}
, \lambda_{25}, \lambda_{26}, \lambda_{27}
$ such that for sufficiently small $\varepsilon>0$, there exist thirteen limit tori in its phase space.
\end{teo}
\begin{proof}
After applying the method described in Section \ref{Sec:method}, we obtain a guiding system \eqref{eq:guidingformula} with a singular point $(\rho,w_0)=(1,0)$ whose associated eigenvalues are $\tau\pm i$. Translating the singular point to the origin and putting the guiding system into its Jordan canonical form yields the following system, for $\tau=0$:
\begin{small}
\begin{equation}\label{eq:ZNiln3quinticJordan}
\begin{aligned}
 x'&=-y+\lambda_{1} x^2-x y+\frac{(271 \lambda_{1}+116 \lambda_{18}-136 \lambda_{25})}{150} x^3 +\left(2 \lambda_{19}-\frac{64 \lambda_{26}}{35}\right)x y^2+\lambda_{20} y^4\\
 &\quad +\frac{ (1001 \lambda_{1}+1996 \lambda_{18}-1816 \lambda_{25})}{1200}x^4+ \left(3 \lambda_{19}-\frac{12 \lambda_{26}}{5}\right)x^2 y^2+\lambda_{18} x^5+\lambda_{19} x^3 y^2+\lambda_{20} x y^4\\
 &\quad +\lambda_{25} x^6+\lambda_{26} x^4 y^2+\frac{3\lambda_{26}}{5}  x^5 y^2+\frac{\lambda_{26}}{5}  x^6 y^2-\frac{3 (143 (\lambda_{1}-4 \lambda_{18})+1912 \lambda_{25})}{6400} x^8+\frac{ \lambda_{26}}{35} x^7 y^2\\
 &\quad -\frac{ (143 \lambda_{1}-572 \lambda_{18}+1812 \lambda_{25})}{2400}x^9-\frac{143 (7 \lambda_{1}-28 \lambda_{18}+88 \lambda_{25})}{38400} x^{10}\\
 &\quad -\frac{13  (7 \lambda_{1}-28 \lambda_{18}+88 \lambda_{25})}{12800}x^{11}-\frac{13 (7 \lambda_{1}-28 \lambda_{18}+88 \lambda_{25})}{76800}x^{12}-\frac{(7 \lambda_{1}-28 \lambda_{18}+88 \lambda_{25})}{76800}x^{13},\\
 y'&=x+\lambda_{11} y^3+\lambda_{9} x^2 y+\lambda_{14} x^3 y+\lambda_{16} x y^3+\lambda_{17} y^4+\lambda_{23} x^2 y^3+\lambda_{24} y^5-\frac{5(\lambda_{16}-2 \lambda_{23})}{6}x^3 y^3\\
 &\quad +\lambda_{27} x^5 y+\frac{77 }{48}x^6+\frac{(102 \lambda_{14}+68 \lambda_{27}-77 \lambda_{9})}{24} x^6 y-\frac{5(\lambda_{16}-2 \lambda_{23})}{8} x^4 y^3 +\frac{11 }{4}x^7+\frac{275}{128}x^8\\
 &\quad +\frac{(57 \lambda_{14}+28 \lambda_{27}-44 \lambda_{9})}{8} x^7 y +\frac{(2 \lambda_{23}-\lambda_{16})}{4} x^5 y^3 +\frac{(354 \lambda_{14}+160 \lambda_{27}-275 \lambda_{9})}{64} x^8 y \\
 &\quad +\frac{(2 \lambda_{23}-\lambda_{16})}{24} x^6 y^3 +\frac{385}{384} x^9+\frac{55(9 \lambda_{14}+4 \lambda_{27}-7 \lambda_{9})}{192} x^9 y +\frac{77}{256}x^{10} +\frac{7}{128} x^{11}+\frac{7 }{1536}x^{12}\\
 &\quad + \frac{11(9 \lambda_{14}+4 \lambda_{27}-7 \lambda_{9})}{128} x^{10} y+\frac{(9 \lambda_{14}+4 \lambda_{27}-7 \lambda_{9})}{64} x^{11} y +\frac{(9 \lambda_{14}+4 \lambda_{27}-7 \lambda_{9})}{768} x^{12} y.
\end{aligned}
\end{equation}
\end{small}
We then turn to study the cyclicity of the origin of system \eqref{eq:ZNiln3quinticJordan}. Denoting by
\[\mu=(\lambda_{1},\lambda_{9},\lambda_{11}, \lambda_{14}, \lambda_{16}, \lambda_{17}
, \lambda_{18}, \lambda_{19}, \lambda_{20}, \lambda_{23}, \lambda_{24}
, \lambda_{25}, \lambda_{26}, \lambda_{27}),\]
notice that the parameter value $\mu=0$ correspond to a point in the Bautin Variety, since for $\mu=0$, system \eqref{eq:ZNiln3quinticJordan} becomes
\[x'=-y-x y,\quad y'=x+\frac{77
   x^6}{48}+\frac{11 x^7}{4}+\frac{275 x^8}{128}+\frac{385 x^9}{384}+\frac{77 x^{10}}{256}+\frac{7 x^{11}}{128}+ \frac{7 x^{12}}{1536},\]
which is time-reversible. Hence, we are set to verify if the hypothesis of Theorem \ref{Teo:ChristopherCenter} are satisfied, by computing $L_1(\mu),\dots,L_{13}(\mu)$. Due to their size, we exhibit only the linear terms of the power series expansion about $\mu=0$ of the first three ones:
\[
\begin{aligned}
L_1^1(\mu)&=\tfrac{221}{75}\lambda_{1}+\tfrac{116 }{75}\lambda_{18}-\tfrac{136}{75} \lambda_{25}+2 \lambda_{11} +\tfrac{4 }{3}\lambda_{19}-\tfrac{128 }{105}\lambda_{26}+\tfrac{2 }{3}\lambda_{9},\\
L_2^1(\mu)&=\tfrac{553}{750}\lambda_{18}+\tfrac{4}{15}\lambda_{14}+\tfrac{152}{75}\lambda_{26}-\tfrac{16 }{15}\lambda_{9}+\frac{2}{5} \lambda_{23}-\tfrac{28}{15}\lambda_{19}-\tfrac{4051}{1000}\lambda_{1}-\tfrac{16}{5} \lambda_{11}+2 \lambda_{24}\\
&\quad +\tfrac{1993}{1125}\lambda_{25},\\
L_3^1(\mu)&=-\tfrac{26101}{15750}\lambda_{18}+\tfrac{73}{126} \lambda_{14}-\tfrac{15668}{3675} \lambda_{26}+\tfrac{1849}{1260}\lambda_{9}-\tfrac{23}{42} \lambda_{23}+\tfrac{62}{15} \lambda_{19}+\tfrac{506551}{63000}\lambda_{1}-\tfrac{17}{84} \lambda_{16}\\
&\quad +\tfrac{751}{105} \lambda_{11}-\tfrac{100}{21}\lambda_{24}+\tfrac{25}{21} \lambda_{27}-\tfrac{42131}{23625} \lambda_{25}.
\end{aligned}
\]
We have that ${\rm rank}\left[\frac{\partial(L_1,\dots,L_8)}{\partial\mu}\right]_{\mu=0}=8$. Performing the change of variables $L_i=u_i$, for $i=1,\dots, 8$, allow us to write the next Lyapunov coefficients, for $u_1=\dots=u_8=0$ as
\[
L_{8+i}(\mu)=h_{8+i}(\lambda_{17}, \lambda_{20}, \lambda_{23}, \lambda_{25}, \lambda_{26}
, \lambda_{27})+O(|\mu|^3),
\]
for $i=1,\dots,5$, where $h_{8+i}$ is a quadratic polynomial. We have omitted their explicit expressions, as their coefficients are rational numbers whose numerators and denominators have more than 40 digits each. The result will follow, by Theorem \ref{Teo:ChristopherCenter}, from the existence of a parameter $\hat{\mu}=(\hat{\lambda}_{17},\hat{\lambda}_{20},\hat{\lambda}_{23},\hat{\lambda}_{25},\hat{\lambda}_{26},\hat{\lambda}_{27})$ such that $h_{9}(\hat{\mu})=\dots=h_{12}(\hat{\mu})=0$, $h_{13}(\hat{\mu})\neq 0$ and ${\rm rank}\left[\frac{\partial(h_{9},\dots,h_{12})}{\partial\mu}\right]_{\mu=\hat{\mu}}=4$. The explicit expression for such a $\hat{\mu}$ contain rational numbers whose fractional representation are too large to fit into the pages of the present paper. Thus, we present a decimal approximation of this expression up to ten digits, given by
\[
\begin{aligned}
\hat{\lambda}_{17}& = \lambda_{17}, \quad &\hat{\lambda}_{20} = 
- 1.711603852 \lambda_{25},\\
\hat{\lambda}_{23} &= - 7.152529612 \lambda_{25}
, \quad &\hat{\lambda}_{25} = \lambda_{25}, \\
\hat{\lambda}_{26} &= 
 7.719711739 \lambda_{25}, \quad &\hat{\lambda}_{27} = - 5.122805170 \lambda_{25}.
\end{aligned} 
 \]
The same situation happens for the explicit expression of $h_{13}(\hat{\mu})$, whose approximation is given by
\[
h_{13}(\hat{\mu})=- 317.0368381\,\lambda_{17}\lambda_{25}.
\]
Thus, by Theorem \ref{Teo:ChristopherCenter}, there are parameter values $(\tau,\mu)$ for which the guiding system \eqref{eq:ZNiln3quinticJordan} has thirteen limit cycles, which in turn implies by Theorem \ref{Teo:InvariantTori} that system \eqref{eq:ZNiln3quintic} has 13 limit tori in its phase space.
\end{proof}

\subsection{A step into second order}\label{Sec:N2>3}

Using second order analysis, we were able to obtain the following result, from which we obtain the improved lower bound on the number of normally hyperbolic limit tori for quadratic three-dimensional systems: $N_h(2)\geqslant 3$.

\begin{teo}\label{N(2)=3}
Consider the following parametric family of differential systems
\begin{small}
\begin{eqnarray}\label{eq:ZHquadratic_ex}
\dot{x}&=&-y+\varepsilon  \left(x^2+\left(\tau ^2-4 \tau +8\right)
   x+\frac{5 x y}{2}-\frac{y^2}{3}-z (z+1)\right)+\frac{\varepsilon ^2}{2}
   \left(2 \tau ^2-8 \tau +15\right) x,\\
\dot{y}&=&x+\varepsilon 
   \left(\frac{1}{2} \left(\beta z^2-2 \left(\tau ^2-4
   \tau +8\right) y\right)+\frac{16 x
   z}{3}\right)+\varepsilon
   ^2 \left(\frac{1}{2} \left(\beta z^2-2 \left(\tau ^2-4
   \tau +8\right) y\right)+\frac{16 x z}{3}\right),\nonumber\\
\dot{z}&=&\frac{\varepsilon}{32}\left(2 x^2+2 x (y+16
   z)-2 y^2+16 y z\right)+\frac{\varepsilon ^2}{32}  \left(-8 (2
   \alpha+3) z^2-\tau ^2+4 \tau +16 (\tau -1)
   z-8\right).\nonumber
\end{eqnarray}
\end{small}
There exist parameter values in a neighborhood of $(\tau,\alpha,\beta)=(0,-\frac{9}{5},-\frac{52}{75})$ such that system \eqref{eq:ZHquadratic_ex} has three normally hyperbolic invariant tori in its phase space for sufficiently small $\varepsilon>0$.
\end{teo}
\begin{proof}
The first averaged function for \eqref{eq:ZHquadratic_ex}, vanishes. The first non-vanishing averaged function is $\mathbf{g}_2(r,w)=(\mathbf{g}_2^1,\mathbf{g}_2^2)$, which by the formulas \eqref{eq:averageformulas} and Propostition \ref{propo:averagedfunctions}, are given by
\begin{equation}
\begin{aligned}
\mathbf{g}_2^1(r,w)&=\frac{1}{2} r \left(r^2-3 w^2-4
   w-1\right),\nonumber\\
\mathbf{g}_2^2(r,w)&=\frac{1}{16} \left(4\tau-8 (2 \alpha+5) w^2-16
   (\beta+1) w^3+r^2 \left(\tau ^2-4 \tau +8\right)-\tau
   ^2 +16 (\tau -1) w-8\right).\nonumber
\end{aligned}
\end{equation}
We have that $(r,w)=(1,0)$ is an equilibrium point of the guiding system $(r',w')=\mathbf{g}_2(r,w)$ such that $D\mathbf{g}_2(1,0)$ has eigenvalues $\frac{\tau}{2}\pm i$.
For $\tau=0$, we have that $(1,0)$ is a Hopf equilibrium point. Assuming $\tau=0$, the coordiante change $u=r+w$, $v=w$ puts the guiding system into its Jordan canonical form:
\begin{eqnarray}
u'&=&-v+\alpha v^2+u^2+\beta
   v^3+\frac{u^3}{2}+\frac{3 u^2 v}{2},\nonumber\\
v'&=&u-\alpha v^2+\frac{u^2}{2}+u v-2 v^2-\beta v^3-v^3,\nonumber
\end{eqnarray}
for which we compute the first five Lyapunov coefficients. We have that
\begin{eqnarray}
L_1&=&-2 \alpha -2 \beta-\frac{4}{3} \alpha^{2} -\frac{2}{3},\nonumber\\
L_2&\equiv&\frac{2 \left(2 \alpha +1\right) \left(2 \alpha -1\right) \left(5 \alpha +9\right) \left(\alpha +2\right)}{45}\;\mod\langle L_1\rangle,\nonumber\\
L_3&\equiv&-\frac{\left(2 \alpha +1\right) \left(2 \alpha -1\right) \left(\alpha +2\right) \left(5140 \alpha^{3}+26468 \alpha^{2}+45619 \alpha +26319\right)}{1890}\;\mod\langle L_1\rangle,\nonumber\\
L_5&\equiv&L_4\equiv 0\;\mod\langle L_1,L_2,L_3\rangle.
\end{eqnarray}
For $(\alpha,\beta)=(-\frac{9}{5},-\frac{52}{75})$, we have $L_1=L_2=0$, $L_3=\frac{19136}{984375}$ and ${\rm rank}\left(\frac{\partial(L_1,L_2)}{\partial(\alpha,\beta)}\right)=2$. By Theorem \ref{Teo:ChristopherFocus}, two limit cycles bifurcate in a neighborhood of the origin for parameters sufficiently close to $(\tau,\alpha,\beta)=(0,-\frac{9}{5},-\frac{52}{75})$. Now, Theorem \ref{Teo:InvariantTori} implies that system \eqref{eq:ZHquadratic_ex} has three normally hyperbolic invariant tori for $\varepsilon>0$ sufficiently small.
\end{proof}

\section*{Acknowledgments}
L.Q. Arakaki is supported by S\~{a}o Paulo Research Foundation (FAPESP) grant 2024/06926-7. L. F. S. Gouveia is supported by UNICAMP's postdoctoral research program (PPPD).  D.D. Novaes is partially supported by S\~{a}o Paulo Research Foundation (FAPESP) grant 2018/13481-0, and by Conselho Nacional de Desenvolvimento Cient\'{i}fico e Tecnol\'{o}gico (CNPq) grant 309110/2021-1.

\section*{Conflict of interest}

\noindent The authors declare that they have no conflict of interest.

\section*{Data availability}

\noindent Data sharing is not applicable to this article as no new data were created or analyzed in this study.

\bibliographystyle{siam}
\bibliography{Referencias.bib}
\end{document}